\documentclass[11pt, reqno]{amsart}
\usepackage{amsfonts}
\usepackage{amsmath}
\usepackage{amssymb, color}
\usepackage{amscd}
\usepackage[mathscr]{eucal}
\usepackage{url}

\allowdisplaybreaks[1]

\newcommand{\ph}[2]{{\left({#1}\right)}_{#2}}

\renewcommand*{\bar}{\overline}
\newcommand{\gfp}[1]{\Gamma_p{\left({#1}\right)}}
\newcommand{\biggfp}[1]{\Gamma_p{\bigl({#1}\bigr)}}
\newcommand{\Biggfp}[1]{\Gamma_p{\Bigl({#1}\Bigr)}}

\theoremstyle{plain}

\newtheorem{theorem}{Theorem}[section]
\newtheorem{lemma}[theorem]{Lemma}

\theoremstyle{definition}
\newtheorem{defi}[theorem]{Definition}

\numberwithin{equation}{section}

\makeatletter
\def\imod#1{\allowbreak\mkern5mu({\operator@font mod}\,\,#1)}
\makeatother

\begin{document}

\title[Hypergeometric type identities in the $p$-adic setting and modular forms]{Hypergeometric type identities in the $p$-adic setting and modular forms}

\author{Jenny G. Fuselier}
\address{Jenny G. Fuselier, Department of Mathematics \& Computer Science\\
Drawer 31, High Point University\\
High Point, NC 27262 \\
USA}
\email{jfuselie@highpoint.edu}

\author{Dermot M\lowercase{c}Carthy}
\address{Dermot M\lowercase{c}Carthy, Department of Mathematics \& Statistics\\
Texas Tech University\\
Lubbock, TX 79410-1042\\
USA}
\email{dermot.mccarthy@ttu.edu}




\begin{abstract}
We prove hypergeometric type identities for a function defined in terms of quotients of the $p$-adic gamma function. We use these identities to prove a supercongruence conjecture of Rodriguez-Villegas between a truncated $_4F_3$ hypergeometric series and the Fourier coefficients of certain weight four modular form.  
\end{abstract}

\maketitle


\section{Introduction}\label{sec_Intro}

Hypergeometric functions over finite fields were first introduced by Greene \cite{G}. Functions of this type were also defined by Katz \cite{K}, about the same time, and more recently by the second author \cite{McC6}. These functions are analogues of classical hypergeometric functions and were first developed to simplify character sum evaluations. Since then they have been applied to a number of areas of mathematics but the two areas of most interest to the authors are their transformation properties \cite{G, McC6}, which often mirror classical hypergeometric transformations, and their connections to modular forms \cite{AO, E, F, F2, F3, FOP, L, L3, M, McC5, McC7}. 

The arguments in finite field hypergeometric functions are multiplicative characters, and, consequently, many of the results involving these functions are restricted to primes in certain congruence classes (see for example \cite{E, F, F2, L, L2, L3, M, V}) to facilitate the existence of characters of specific orders.

To overcome these restrictions, the second author defined the following function \cite{McC4, McC8} which extends finite field hypergeometric functions to the $p$-adic setting, and can often allow results involving finite field hypergeometric functions to be extended to a wider class of primes \cite{McC4, McC5, McC8}. 
Let $\mathbb{F}_{p}$ denote the finite field with $p$, a prime, elements. Let $\gfp{\cdot}$ denote Morita's $p$-adic gamma function and let $\omega$ denote the Teichm\"{u}ller character of $\mathbb{F}_p$ with $\bar{\omega}$ denoting its character inverse. 
For $x \in \mathbb{Q}$ we let  $\left\lfloor x \right\rfloor$ denote the greatest integer less than or equal to $x$ and
$\langle x \rangle$ denote the fractional part of $x$, i.e. $x- \left\lfloor x \right\rfloor$.
\begin{defi}\cite[Definition 1.1]{McC8}\label{def_Gp}
Let $p$ be an odd prime and let $s \in \mathbb{F}_p$. For $n \in \mathbb{Z}^{+}$ and $1 \leq i \leq n$, let $a_i, b_i \in \mathbb{Q} \cap \mathbb{Z}_p$.
Then we define  
\begin{multline*}
{_{n}G_{n}}
\biggl[ \begin{array}{cccc} a_1, & a_2, & \dotsc, & a_n \\
 b_1, & b_2, & \dotsc, & b_n \end{array}
\Big| \; s \; \biggr]_p
: = \frac{-1}{p-1}  \sum_{j=0}^{p-2} 
(-1)^{jn}\;
\bar{\omega}^j(s)\\
\times \prod_{i=1}^{n} 
\frac{\biggfp{\langle a_i -\frac{j}{p-1}\rangle}}{\biggfp{\langle a_i \rangle}}
\frac{\biggfp{\langle - b_i +\frac{j}{p-1}\rangle}}{\biggfp{\langle -b_i \rangle}}
(-p)^{-\lfloor{\langle a_i \rangle -\frac{j}{p-1}}\rfloor -\lfloor{\langle -b_i \rangle +\frac{j}{p-1}}\rfloor}.
\end{multline*}
\end{defi}
\noindent Throughout the paper we will refer to this function as ${_{n}G_{n}}[\cdots]$.
We note that the value of ${_{n}G_{n}}[\cdots]$ depends only on the fractional part of the $a$ and $b$ parameters, and is invariant if we change the order of the parameters.

Results involving finite field hypergeometric functions can readily be converted to expressions involving ${_{n}G_{n}}[\cdots]$ using Lemma \ref{lem_G_to_F}, which we will see later. However these new expressions in ${_{n}G_{n}}[\cdots]$ will still only be valid for primes $p$ where the original characters existed over $\mathbb{F}_p$, i.e., for primes in certain congruence classes. It is a non-trivial exercise to then extend these results to almost all primes. 

The purpose of this paper is threefold. Firstly, we establish certain transformations and identities for ${_{n}G_{n}}[\cdots]$ in full generality, i.e., for almost all primes and beyond expressions which can be implied directly by any existing transformations for finite field hypergeometric functions. Secondly, we use these transformations to prove a linear relation between a certain ${_{4}G_{4}}[\cdots]_p$ and the $p$-th Fourier coefficients of a weight four modular form. As a result, we prove a supercongruence conjecture of Rodriguez-Villegas between a truncated $_4F_3$ classical hypergeometric series and the $p$-th Fourier coefficients of the same weight four modular form, modulo $p^3$.

The rest of this paper is organized as follows. In the next section we expand on the discussion above and state our results. In Section \ref{sec_Prelim} we recall some basic properties of multiplicative characters, Gauss sums and the $p$-adic gamma function.
We discuss the role of finite field hypergeometric functions in proving supercongruences, and their relationship to ${_{n}G_{n}}[\cdots]$, in Section \ref{sec_FFHypFn}. The proofs of our main results are contained in Section \ref{sec_Proofs}. Finally, we make some closing remarks in Section \ref{sec_cr}.


\section{Statement of Results}\label{sec_Results}

We first define the truncated generalized hypergeometric series. For a complex number $a$ and a non-negative integer $n$ let $\ph{a}{n}$ denote the rising factorial defined by
\begin{equation*}\label{RisFact}
\ph{a}{0}:=1 \quad \textup{and} \quad \ph{a}{n} := a(a+1)(a+2)\dotsm(a+n-1) \textup{ for } n>0.
\end{equation*}
Then, for complex numbers $a_i$, $b_j$ and $z$, with none of the $b_j$ being negative integers or zero, we define the truncated generalized hypergeometric series
\begin{equation*}
{{_rF_s} \left[ \begin{array}{ccccc} a_1, & a_2, & a_3, & \dotsc, & a_r \vspace{.05in}\\
\phantom{a_1} & b_1, & b_2, & \dotsc, & b_s \end{array}
\Big| \; z \right]}_{m}
:=\sum^{m}_{n=0}
\frac{\ph{a_1}{n} \ph{a_2}{n} \ph{a_3}{n} \dotsm \ph{a_r}{n}}
{\ph{b_1}{n} \ph{b_2}{n} \dotsm \ph{b_s}{n}}
\; \frac{z^n}{{n!}}.
\end{equation*}
When we drop the subscript $m$ we will use the same notation to mean the generalized hypergeometric series, which sums to infinity.

In \cite{R} Rodriguez-Villegas examined the relationship between the number of points over $\mathbb{F}_p$ on certain Calabi-Yau manifolds and truncated generalized hypergeometric series which correspond to a particular period of the manifold. In doing so, he identified numerically 14 possible supercongruences between truncated $_4F_3$ hypergeometric series and Fourier coefficients of certain weight four modular forms, modulo $p^3$. To date, two of these conjectures have been proven \cite{Ki,McC5}.

The first one of these supercongruence conjectures proven, by Kilbourn \cite{Ki}, can be described as follows. Let 
\begin{equation*}\label{def_f1}
f_1(z):= q \prod_{n=1}^{\infty}(1-q^{2n})^4 (1-q^{4n})^4 = \sum_{n=1}^{\infty} a(n)q^n,
\end{equation*}
where $q=e^{2 \pi i z}$, be the unique newform in $S_4(\Gamma_0(8))$, the space of weight four cusp forms on the congruence subgroup $\Gamma_0(8)$. Then, for all odd primes $p$,
\begin{equation}\label{thm_Kilbourn}
{_{4}F_3} \Biggl[ \begin{array}{cccc} \frac{1}{2}, & \frac{1}{2}, & \frac{1}{2}, & \frac{1}{2}\vspace{.05in}\\
\phantom{\frac{1}{2}} & 1, & 1, & 1 \end{array}
\bigg| \; 1 \Biggr]_{p-1}
\equiv
a(p)
\pmod{p^3}.
\end{equation}

\noindent Kilbourn's proof of (\ref{thm_Kilbourn}) can be interpreted in terms of finite field hypergeometric functions. However, by definition, results involving finite field hypergeometric functions will often be restricted to primes in certain congruence classes (see for example \cite{E, F, F2, L, L2, L3, M, V}). Consequently, based on the parameters involved, it appears that these functions cannot be used to fully establish the other 13 supercongruence conjectures of Rodriguez-Villegas. (Please see Section \ref{sec_FFHypFn} for a detailed discussion of finite field hypergeometric functions and their restrictions.) To overcome these restrictions, the second author defined ${_{n}G_{n}}[\cdots]$ (see Definition \ref{def_Gp}), which extends finite field hypergeometric functions to the $p$-adic setting, and can often allow results involving finite field hypergeometric functions to be extended to a wider class of primes \cite{McC4, McC5, McC8}.

Families of congruences between ${_{n}G_{n}}[\cdots]$ and truncated generalized hypergeometric series have been established in \cite{McC4} and these congruences cover all 14 series listed in Rodriguez-Villegas' conjectures. Therefore, to establish the remaining conjectures, it suffices to link certain ${_{n}G_{n}}[\cdots]$ to the Fourier coefficients of the relevant modular forms.
Using this method, a second of Rodriguez-Villegas' supercongruence conjectures was proved by the second author in \cite{McC5}. (Please see Section \ref{sec_FFHypFn} and (\ref{thm_DMCsuper2}) for further details.) 
One such family of congruences is as follows.
\begin{theorem}\cite[Theorem 2.6]{McC4}\label{thm_4G1}
Let $d_1, d_2 \geq 2$ be integers and let $p$ be an odd prime such that $p\equiv \pm1 \imod {d_1}$ and $p\equiv \pm1 \imod {d_2}$. If $s(p):=\biggfp{\frac{1}{d_1}}\biggfp{\frac{d_1-1}{d_1}}\biggfp{\frac{1}{d_2}}\biggfp{\frac{d_2-1}{d_2}}=(-1)^{\left\lfloor \frac{p-1}{d_1} \right\rfloor+\left\lfloor \frac{p-1}{d_2} \right\rfloor}$, then
\begin{multline*}
{_{4}G_{4}}
\biggl[ \begin{array}{cccc} \frac{1}{d_1} & 1-\frac{1}{d_1} & \frac{1}{d_2} & 1-\frac{1}{d_2} \\
 1 & 1 & 1 & 1 \end{array}
\Big| \; 1 \; \biggr]_p\\
\equiv
{_{4}F_3} \Biggl[ \begin{array}{cccc} \frac{1}{d_1}, & 1-\frac{1}{d_1}, & \frac{1}{d_2}, & 1-\frac{1}{d_2}\\
\phantom{\frac{1}{d_1}} & 1, & 1, & 1 \end{array}
\bigg| \; 1 \Biggr]_{p-1}
+s(p)\hspace{1pt}p
\pmod {p^3}.
\end{multline*}
\end{theorem}

The purpose of this paper is to establish certain identities for ${_{n}G_{n}}[\cdots]$ which then can be used to establish another one of Rodriguez-Villegas' supercongruence conjectures, as follows. 
\begin{theorem}\label{thm_RV3}
Let $f_2 \in S_4(\Gamma_0(16))$ be the unique weight 4 newform of level 16, with Fourier expansion $ f_2(z):= \sum_{n=1}^{\infty} c(n)q^n$, where $q=e^{2 \pi i z}$. Then, for all odd primes $p$,
\begin{equation*}
{_{4}F_3} \Biggl[ \begin{array}{cccc} \frac{1}{2}, & \frac{1}{2}, & \frac{1}{4}, & \frac{3}{4}\vspace{.05in}\\
\phantom{\frac{1}{2}} & 1, & 1, & 1 \end{array}
\bigg| \; 1 \Biggr]_{p-1}
\equiv
c(p)
\pmod{p^3}.
\end{equation*}
\end{theorem}
\noindent This congruence was established by Lennon \cite{L} in the case $p \equiv 1 \pmod4$, using finite field hypergeometric functions. 
Appealing to Theorem \ref{thm_4G1} with $d_1=2$, $d_2=4$, it suffices to prove
\begin{theorem}\label{thm_GtoC}
Let $p$ be an odd prime and define $s(p)=\biggfp{\frac{1}{4}}\biggfp{\frac{3}{4}}\biggfp{\frac{1}{2}}^2$. Then
\begin{equation*}
{_{4}G_{4}}
\biggl[ \begin{array}{cccc} \frac{1}{2} & \frac{1}{2} & \frac{1}{4} & \frac{3}{4} \\
 1 & 1 & 1 & 1 \end{array}
\Big| \; 1 \; \biggr]_p
- s(p) \cdot p
=c(p).
\end{equation*}
\end{theorem}

\noindent Like Lennon, we will rely on the fact that $f_2$ is quadratic twist of $f_1$ with $c_p = \phi_p(-1) a_p$, where $\phi_p(\cdot)$ is the Legendre symbol modulo $p$.  As we will see later, combining aspects of Kilbourn's proof of (\ref{thm_Kilbourn}) (specifically (\ref{thm_AO}))  with properties of ${_{n}G_{n}}[\cdots]$ (Lemma \ref{lem_G_to_F})
we know that, for all odd primes $p$,
\begin{equation*}
{_{4}G_{4}}
\biggl[ \begin{array}{cccc} \frac{1}{2} & \frac{1}{2} & \frac{1}{2} & \frac{1}{2} \\
 1 & 1 & 1 & 1 \end{array}
\Big| \; 1 \; \biggr]_p
-p=a(p).
\end{equation*}
Consequently, the proof of Theorem \ref{thm_GtoC}, and hence Theorem \ref{thm_RV3}, can be reduced to establishing the following identity.
\begin{theorem}\label{thm_GMainId}
Let $p$ be an odd prime and define $s(p):=\biggfp{\frac{1}{4}}\biggfp{\frac{3}{4}}\biggfp{\frac{1}{2}}^2=(-1)^{\left\lfloor \frac{p-1}{4} \right\rfloor+\left\lfloor \frac{p-1}{2} \right\rfloor}$. Then
\begin{equation*}
{_{4}G_{4}}
\biggl[ \begin{array}{cccc} \frac{1}{2} & \frac{1}{2} & \frac{1}{4} & \frac{3}{4} \\
 1 & 1 & 1 & 1 \end{array}
\Big| \; 1 \; \biggr]_p
- s(p) \cdot p
=\omega^{\frac{p-1}{2}}(-1)
\left[ \,
{_{4}G_{4}}
\biggl[ \begin{array}{cccc} \frac{1}{2} & \frac{1}{2} & \frac{1}{2} & \frac{1}{2} \\
 1 & 1 & 1 & 1 \end{array}
\Big| \; 1 \; \biggr]_p
-p \,
\right],
\end{equation*}
where $\omega$ is the Teichm\"{u}ller character.
\end{theorem}
\noindent We will need the following transformations to establish Theorem \ref{thm_GMainId}.
Let $\delta(x)$ equal $1$ if $x=0$, and equal zero otherwise.
\begin{theorem}\label{thm_3F2QuadTrans}
Let $p$ be an odd prime and define $s(p):=\biggfp{\frac{1}{4}}\biggfp{\frac{3}{4}}\biggfp{\frac{1}{2}}^2=(-1)^{\left\lfloor \frac{p-1}{4} \right\rfloor+\left\lfloor \frac{p-1}{2} \right\rfloor}$. For $1 \neq x \in \mathbb{F}_p^{*}$,
\begin{multline*}
{_{3}G_{3}}
\biggl[ \begin{array}{cccc} \frac{1}{2} & \frac{1}{2} & \frac{1}{2} \\
 1 & 1 & 1 \end{array}
\Big| \; x^{-1} \; \biggr]_p \\
=
\delta(1+x)  \cdot p \cdot \omega^{\frac{p-1}{2}}(-1)
+ \, s(p)  \cdot \omega^{\frac{p-1}{2}} \left(2(1-x)\right)  \cdot
{_{3}G_{3}}
\biggl[ \begin{array}{cccc} \frac{1}{4} & \frac{3}{4} & \frac{1}{2} \\
 1 & 1 & 1 \end{array}
\Big| \; -\frac{(1-x)^2}{4x}  \; \biggr]_p,
\end{multline*}
where $\omega$ is the Teichm\"{u}ller character.
\end{theorem}

\begin{theorem}\label{thm_GReduction}
Let $p$ be an odd prime and let $s \in \mathbb{F}_p$. For $a_i \in \mathbb{Q} \cap \mathbb{Z}_p$, $1 \leq i \leq n$,
\begin{equation*}
{_{n+1}G_{n+1}}
\biggl[ \begin{array}{ccccc} a_1 & a_2 & \dotsc & a_n & \frac{1}{2} \\
 1 & 1 & \dotsc & 1 & 1 \end{array}
\Big| \; s \; \biggr]_p
=
-
\sum_{t=2}^{p-1}
{_{n}G_{n}}
\biggl[ \begin{array}{cccc} a_1 & a_2 & \dotsc & a_n \\
 1 & 1 & \dotsc & 1 \end{array}
\Big| \; st \; \biggr]_p
\cdot \; \omega^{\frac{p-1}{2}}(1-t),
\end{equation*}
where $\omega$ is the Teichm\"{u}ller character.
\end{theorem}

\begin{theorem}\label{thm_GInversion}
Let $2 \leq n \in \mathbb{Z}$. For $p$ an odd prime and $s \in \mathbb{F}_p^{*}$,
\begin{equation*}
{_{n}G_{n}}
\biggl[ \begin{array}{cccccc} \frac{1}{2} &  \frac{1}{2} & \dotsc &  \frac{1}{2} & \frac{1}{4} & \frac{3}{4} \vspace{.05in}\\
 1 & 1 & \dotsc & 1 & 1 & 1 \end{array}
\Big| \; s \; \biggr]_p
=
p \cdot \omega^{\frac{p-1}{2}}((-1)^{n+1}s) \cdot
{_{n}G_{n}}
\biggl[ \begin{array}{cccccc} \frac{1}{2} &  \frac{1}{2} & \dotsc &  \frac{1}{2} & \frac{1}{2} & \frac{1}{2} \vspace{.05in}\\
 1 & 1 & \dotsc & 1 & \frac{1}{4} & \frac{3}{4} \end{array}
\Big| \; s^{-1} \; \biggr]_p
\end{equation*}
where $\omega$ is the Teichm\"{u}ller character.
\end{theorem}

Theorem \ref{thm_3F2QuadTrans} is an analogue of the classical hypergeometric quadratic transformation due to Whipple \cite[(3.1.15) page 130]{AAR}
\begin{multline*}
{_3F_2} \left[ \begin{array}{ccc} a & b & c \vspace{.05in}\\
\phantom{a} & 1+a-b & 1+a-c \end{array}
\Big| \; z \right]\\
=
(1-z)^{-a} \,
{_3F_2} \left[ \begin{array}{ccc}  \frac{a}{2} & \frac{a+1}{2} & 1+a-b-c \vspace{.05in}\\
\phantom{a} & 1+a-b & 1+a-c \end{array}
\Big| \; -\frac{4z}{(1-z)^2} \right],
\end{multline*}
when $a=b=c=\frac{1}{2}$.

Many transformations exist for finite field hypergeometric functions which are analogues of classical results \cite{G, McC6}. These transformations can readily be converted to expressions involving ${_{n}G_{n}}[\cdots]$ using Lemma \ref{lem_G_to_F}. However the resulting transformations in ${_{n}G_{n}}[\cdots]$ will only be valid for primes $p$ where the original characters existed over $\mathbb{F}_p$, i.e., for primes in certain congruence classes. It is a non-trivial exercise to then extend these results to almost all primes. There are very few identities and transformations for ${_{n}G_{n}}[\cdots]$ in full generality. Barman and Saikia \cite{BS1, BS2} provide transformations for ${_{2}G_{2}}[\cdots]$, by counting points on various families of elliptic curves. To our knowledge Theorems \ref{thm_GMainId} to \ref{thm_GInversion} are the only other full ${_{n}G_{n}}[\cdots]$ identities.

Theorems \ref{thm_3F2QuadTrans} to \ref{thm_GInversion} are extended versions of special cases of the finite field hypergeometric function results described in Corollary 4.30, Theorem 3.13 and Theorem 4.2(b) respectively in \cite{G}.


\section{Preliminaries}\label{sec_Prelim}

Let $\mathbb{Z}_p$ denote the ring of $p$-adic integers, $\mathbb{Q}_p$ the field of $p$-adic numbers, $\bar{\mathbb{Q}_p}$ the algebraic closure of $\mathbb{Q}_p$, and $\mathbb{C}_p$ the completion of $\bar{\mathbb{Q}_p}$.
\subsection{Multiplicative Characters and Gauss Sums}
Let $\widehat{\mathbb{F}^{*}_{p}}$ denote the group of multiplicative characters of $\mathbb{F}^{*}_{p}$. 
We extend the domain of $\chi \in \widehat{\mathbb{F}^{*}_{p}}$ to $\mathbb{F}_{p}$, by defining $\chi(0):=0$ (including the trivial character $\varepsilon$) and denote $\bar{\chi}$ as the inverse of $\chi$. When $p$ is odd we denote the character of order 2 of $\mathbb{F}_p^*$ by $\phi_p$. We will often drop the subscript $p$ if it is clear from the context.

We now introduce some properties of Gauss sums. For further details see \cite{BEW}, noting that we have adjusted results to take into account $\varepsilon(0)=0$. 
Let $\zeta_p$ be a fixed primitive $p$-th root of unity in $\bar{\mathbb{Q}_p}$. We define the additive character $\theta : \mathbb{F}_p \rightarrow \mathbb{Q}_p(\zeta_p)$ by $\theta(x):=\zeta_p^{x}$. We note that $\mathbb{Q}_p$ contains all $(p-1)$-th roots of unity and in fact they are all in $\mathbb{Z}^{*}_p$. Thus we can consider multiplicative characters of $\mathbb{F}_p^{*}$ to be maps $\chi: \mathbb{F}_p^{*} \to \mathbb{Z}_{p}^{*}$. 
Recall then that for $\chi \in \widehat{\mathbb{F}_p^{*}}$, the Gauss sum $g(\chi)$ is defined by 
$g(\chi):= \sum_{x \in \mathbb{F}_p} \chi(x) \theta(x).$

The following important result gives a simple expression for the product of two Gauss sums. For $\chi \in \widehat{\mathbb{F}_p^{*}}$ we have
\begin{equation}\label{for_GaussConj}
g(\chi)g(\bar{\chi})=
\begin{cases}
\chi(-1) p & \text{if } \chi \neq \varepsilon,\\
1 & \text{if } \chi= \varepsilon.
\end{cases}
\end{equation}
\noindent Another important product formula for Gauss sums is the Hasse-Davenport formula.
\begin{theorem}[Hasse, Davenport \cite{BEW} Thm 11.3.5]\label{thm_HD}
Let $\psi$ be a character of order $m$ of $\mathbb{F}_p^*$ for some positive integer $m$. For a character $\chi$ of $\mathbb{F}_p^*$ we have
\begin{equation*}
\prod_{i=0}^{m-1} g(\psi^i \chi) = g(\chi^m) \chi^{-m}(m)\prod_{i=1}^{m-1} g(\psi^i).
\end{equation*}
\end{theorem}

\noindent Recall also that for $\chi, \psi \in \widehat{\mathbb{F}^{*}_{p}}$ we define the Jacobi sum by $J(\chi,\psi):=\sum_{t \in \mathbb{F}_p} \chi(t) \psi(1-t).$
\noindent We can relate Jacobi sums to Gauss sums.
For $\chi$, $\psi \in \widehat{\mathbb{F}_q^{*}}$ not both trivial,
\begin{equation}\label{for_JactoGauss}
J(\chi, \psi)=
\begin{cases}
\dfrac{g(\chi)g(\psi)}{g(\chi \psi)}
& \qquad \text{if } \chi \psi \neq \varepsilon,\\[18pt]
-\dfrac{g(\chi)g(\psi)}{p}
&\qquad \text{if }\chi \psi = \varepsilon \: .
\end{cases}
\end{equation}

\subsection{$p$-adic Preliminaries}\label{subsec_padicPrelim}
We define the Teichm\"{u}ller character to be the primitive character $\omega: \mathbb{F}_p \rightarrow\mathbb{Z}^{*}_p$ satisfying $\omega(x) \equiv x \pmod p$ for all $x \in \{0,1, \ldots, p-1\}$.
We now recall the $p$-adic gamma function. For further details, see \cite{Ko}.
Let $p$ be an odd prime.  For $n \in \mathbb{Z}^{+}$ we define the $p$-adic gamma function as
\begin{align*}
\gfp{n} &:= {(-1)}^n \prod_{\substack{0<j<n\\p \nmid j}} j \\
\intertext{and extend to all $x \in\mathbb{Z}_p$ by setting $\gfp{0}:=1$ and} 
\gfp{x} &:= \lim_{n \rightarrow x} \gfp{n}
\end{align*}
for $x\neq 0$, where $n$ runs through any sequence of positive integers $p$-adically approaching $x$. 
This limit exists, is independent of how $n$ approaches $x$, and determines a continuous function
on $\mathbb{Z}_p$ with values in $\mathbb{Z}^{*}_p$.
We now state a product formula for the $p$-adic gamma function.
If $m\in\mathbb{Z}^{+}$, $p \nmid m$ and $x=\frac{r}{p-1}$ with $0\leq r \leq p-1$ then
\begin{equation}\label{for_pGammaMult}
\prod_{h=0}^{m-1} \gfp{\tfrac{x+h}{m}}=\omega\left(m^{(1-x)(1-p)}\right)
\gfp{x} \prod_{h=1}^{m-1} \gfp{\tfrac{h}{m}}.
\end{equation}
We note also that
\begin{equation}\label{for_pGammaOneMinus}
\gfp{x}\gfp{1-x} = {(-1)}^{x_0},
\end{equation}
where $x_0 \in \{1,2, \dotsc, {p}\}$ satisfies $x_0 \equiv x \pmod {p}$.
\noindent The Gross-Koblitz formula \cite{GK} allows us to relate Gauss sums and the $p$-adic gamma function. Let $\pi \in \mathbb{C}_p$ be the fixed root of $x^{p-1}+p=0$ which satisfies ${\pi \equiv \zeta_p-1 \pmod{{(\zeta_p-1)}^2}}$. Then we have the following result.
\begin{theorem}[Gross, Koblitz \cite{GK}]\label{thm_GrossKoblitz}
For $ j \in \mathbb{Z}$,
\begin{equation*} 
g(\bar{\omega}^j)=-\pi^{(p-1) \langle{\frac{j}{p-1}}\rangle} \: \gfp{\langle{\tfrac{j}{p-1}}\rangle}.
\end{equation*}
\end{theorem}

We will need the following lemmas in the proofs of Theorems \ref{thm_GMainId} and \ref{thm_GReduction} respectively.
\begin{lemma}\label{Quad_Sum}
For $p$ an odd prime and $0 < j < p-1 \in \mathbb{Z}$,
\begin{equation*}
\frac{\biggfp{\langle \frac{1}{2} -\frac{j}{p-1}\rangle}\biggfp{\frac{j}{p-1}}}{\biggfp{\frac{1}{2}}}
(-p)^{-\lfloor{\frac{1}{2} -\frac{j}{p-1}}\rfloor} 
=
-\sum_{t=2}^{p-1} \; \bar{\omega}^j\left(\frac{4(1-t)}{t^2}\right)
\end{equation*}
\end{lemma}

\begin{proof}
Let $\chi \in \widehat{\mathbb{F}^{*}_{p}}$. By the Hasse-Davenport formula, Theorem \ref{thm_HD}, with $m=2$ we see that
\begin{equation}\label{HD_phi}
g(\phi \chi) \, g(\chi) = g(\chi ^2) \, g(\phi) \, \bar{\chi}(4).
\end{equation}
Applying (\ref{HD_phi}) and (\ref{for_JactoGauss}) we get that, for $\chi \neq \varepsilon$,
\begin{equation*}
\frac{g(\phi \chi) g(\bar{\chi})}{g(\phi)} = \frac{g(\chi^2) g(\bar{\chi})}{g(\chi)} \bar{\chi}(4) = J(\chi^2, \bar{\chi}) \; \bar{\chi}(4)
\end{equation*}
Now consider the Jacobi sum
\begin{align*}
J(\chi^2, \bar{\chi})  = \sum_{t=2}^{p-1} \; \chi^2(t) \,\bar{\chi}(1-t) = \sum_{t=2}^{p-1} \; \chi \left(\frac{t^2}{1-t}\right).
\end{align*}
Therefore, for $\chi \neq \varepsilon$,
\begin{equation*}\label{for_lemQuad_Sum1}
\frac{g(\phi \chi) g(\bar{\chi})}{g(\phi)}  =  \sum_{t=2}^{p-1} \; \chi \left(\frac{t^2}{4(1-t)}\right).
\end{equation*}
Letting $\chi = {\omega}^j$, and using the Gross-Koblitz formula, Theorem \ref{thm_GrossKoblitz}, to convert the left-hand side to an expression in $p$-adic gamma functions, completes the proof. 
\end{proof}

\begin{lemma}\label{NonQuad_Sum}
For $p$ an odd prime and $0 \leq j < p-1 \in \mathbb{Z}$,
\begin{equation*}
\frac{\biggfp{\langle \frac{1}{2} -\frac{j}{p-1}\rangle}\biggfp{\frac{j}{p-1}}}{\biggfp{\frac{1}{2}}}
(-p)^{-\lfloor{\frac{1}{2} -\frac{j}{p-1}}\rfloor} 
=
-\sum_{t=2}^{p-1} \; \omega^j (-t)\; \omega^{\frac{p-1}{2}} (t(t-1))
\end{equation*}
\end{lemma}

\begin{proof}
Let $\chi \in \widehat{\mathbb{F}^{*}_{p}}$. Using (\ref{for_GaussConj}) and (\ref{for_JactoGauss}), we get that for $\chi \neq \varepsilon$,
\begin{equation*}\label{for_lemQuad_Sum2}
\frac{g(\phi \chi) g(\bar{\chi})}{g(\phi)}
= \frac{g(\phi \chi) g(\phi) \, \chi(-1)}{g(\chi) \, \phi(-1)}
= \phi\chi(-1) J(\phi\chi, \phi)
= \sum_{t=2}^{p-1} \; \chi(-t) \, \phi(t(t-1))
\end{equation*}
Again using (\ref{for_GaussConj}) and (\ref{for_JactoGauss}) we see that
\begin{equation*}
\phi(-1) J(\phi, \phi) = -\frac{\phi(-1) g(\phi) g(\phi)}{p} = -1 = \frac{g(\phi) g(\varepsilon)}{g(\phi)}
\end{equation*}
Therefore, for all $\chi \in \widehat{\mathbb{F}^{*}_{p}}$,
\begin{equation*}
\frac{g(\phi \chi) g(\bar{\chi})}{g(\phi)}
= \phi\chi(-1) J(\phi\chi, \phi)
= \sum_{t=2}^{p-1} \; \chi(-t) \, \phi(t(t-1)).
\end{equation*}
Letting $\chi = {\omega}^j$, and using the Gross-Koblitz formula, Theorem \ref{thm_GrossKoblitz}, to convert the left-hand side to an expression in $p$-adic gamma functions, completes the proof. 
\end{proof}


\section{Finite Field Hypergeometric Functions}\label{sec_FFHypFn}

Hypergeometric functions over finite fields were originally defined by Greene \cite{G}, who first established these functions as analogues of classical hypergeometric functions.
Functions of this type were also introduced by Katz \cite{K} about the same time.
In the present article we use a normalized version of these functions defined by the second author in \cite{McC6}, which is more suitable for our purposes. The reader is directed to \cite[\S 2]{McC6} for the precise connections among these three classes of functions.

\begin{defi}\cite[Definition 1.4]{McC6}\label{def_F} 
For $A_0,A_1,\dotsc, A_n, B_1, \dotsc, B_n \in \widehat{\mathbb{F}_p^{*}}$ and $x \in \mathbb{F}_{p}$ define
\begin{multline*}\label{def_HypFnFF}
{_{n+1}F_{n}} {\biggl( \begin{array}{cccc} A_0 & A_1 & \dotsc, & A_n \\
 \phantom{A_0} & B_1 & \dotsc, & B_n \end{array}
\Big| \; x \biggr)}_{p}\\
:= \frac{1}{p-1}  \sum_{\chi \in \widehat{\mathbb{F}_p^{*}}} 
\prod_{i=0}^{n} \frac{g(A_i \chi)}{g(A_i)}
\prod_{j=1}^{n} \frac{g(\bar{B_j \chi})}{g(\bar{B_j})}
 g(\bar{\chi})
 \chi(-1)^{n+1}
 \chi(x).
 \end{multline*}
\end{defi}
\noindent Many of the results concerning hypergeometric functions over finite fields that we quote from other articles were originally stated using Greene's function. 
If this is the case, note then that we have reformulated them in terms ${_{n+1}F_{n}}(\cdots)$ as defined above.  

Kilbourn \cite{Ki} proved that, if $p$ is an odd prime, then
\begin{equation*}\label{thm_Kilbourn1}
{_{4}F_3} \Biggl[ \begin{array}{cccc} \frac{1}{2} & \frac{1}{2} & \frac{1}{2} & \frac{1}{2}\\
\phantom{\frac{1}{2}} & 1 & 1 & 1 \end{array}
\bigg| \; 1 \Biggr]_{p-1}
\equiv
{_{4}F_3}  \biggl( \begin{array}{cccc} \phi & \phi & \phi & \phi \\
\phantom{\phi} & \varepsilon & \varepsilon & \varepsilon \end{array}
\bigg| \; 1 \biggr)_{p}-p \pmod{p^3} .
\end{equation*}
He combined this with the following result of Ahlgren and Ono \cite{AO} to establish (\ref{thm_Kilbourn}):
\begin{equation}\label{thm_AO}
  {_{4}F_3}  \biggl( \begin{array}{cccc} \phi & \phi & \phi & \phi \\
\phantom{\phi} & \varepsilon & \varepsilon & \varepsilon \end{array}
\bigg| \; 1 \biggr)_{p}-p
=a(p).
\end{equation}
\noindent 
The main method of relating finite field hypergeometric functions and Fourier coefficients of modular forms has been via the Eichler-Selberg trace formula \cite{AO, F, F2, F3, FOP, L, L3, McC7}. But, apart from a number of special cases, including Ahlgren and Ono's relation above, most of these result are either restricted to primes in certain congruence classes to facilitate the existence of characters of certain orders, which appear as arguments in the finite field hypergeometric functions, or require much more complex relations than (\ref{thm_AO}). This makes finite field hypergeometric functions unsuitable for establishing the remaining Rodriguez-Villegas' supercongruence conjectures in full generality.

For example, let 
\begin{equation*}\label{for_ModForm}
g(z):= g_1(z)+5g_2(z)+20g_3(z)+25g_4(z)+25g_5(z)=\sum_{n=1}^{\infty} b(n) q^n
\end{equation*}
where $g_i(z):=\eta^{5-i}(z) \hspace{2pt} \eta^4(5z) \hspace{2pt} \eta^{i-1}(25z)$,
$\eta(z):=q^{\frac{1}{24}} \prod_{n=1}^{\infty}(1-q^n)$ is the Dedekind eta function and $q:=e^{2 \pi i z}$. Then $f$ is a newform of weight four on the congruence subgroup $\Gamma_0(25)$, and when $p\equiv 1 \pmod 5$, we get that \cite{McC4, McC5}
\begin{equation}\label{thm_DMCsuper}
{_{4}F_3} \Biggl[ \begin{array}{cccc} \frac{1}{5} & \frac{2}{5} & \frac{3}{5} & \frac{4}{5}\\
\phantom{\frac{1}{5}} & 1 & 1 & 1 \end{array}
\bigg| \; 1 \Biggr]_{p-1}
\stackrel{\pmod {p^3}} \equiv
{_{4}F_3}  \biggl( \begin{array}{cccc} \psi & \psi^2 & \psi^3 & \psi^4 \\
\phantom{\psi} & \varepsilon & \varepsilon & \varepsilon \end{array}
\bigg| \; 1 \biggr)_{p}-p
= b(p)
\end{equation}
where $\psi \in \widehat{\mathbb{F}_p^{*}}$ is a character of order 5. Because the finite field hypergeometric function above is based on characters of order $5$, we know that the result cannot be extended beyond $p \equiv 1 \pmod 5$ using this function.
Extending (\ref{thm_DMCsuper}) to all primes, and proving another supercongruence conjecture of Rodriguez-Villegas, was the motivation of the second author to develop ${_{n}G_{n}}[\cdots]$.
We have the following relationship between ${_{n}G_{n}}[\cdots]$ and ${_{n+1}F_{n}}(\cdots)$.
\begin{lemma}\cite[Lemma 3.3]{McC8}\label{lem_G_to_F}
For a fixed odd prime $p$, let $A_i, B_k \in \widehat{\mathbb{F}_p^{*}}$ be given by $\bar{\omega}^{a_i(p-1)}$ and $\bar{\omega}^{b_k(p-1)}$ respectively, where $\omega$ is the Teichm\"{u}ller character, and let $s \in \mathbb{F}_p$ . Then
\begin{equation*}
{_{n+1}F_{n}} {\biggl( \begin{array}{cccc} A_0 & A_1 & \dotsc & A_n \\
 \phantom{B_0} & B_1 & \dotsc & B_n \end{array}
\Big| \; s \biggr)}_{p}
=
{_{n+1}G_{n+1}}
\biggl[ \begin{array}{cccc} a_0 & a_1 & \dotsc & a_n \\
 1 & b_1 & \dotsc & b_n \end{array}
\Big| \; s^{-1} \; \biggr]_p.
\end{equation*}

\end{lemma}

\noindent This relation opens the possibility of extending results involving finite field hypergeometric functions to all but finitely many primes by using ${_{n}G_{n}}[\cdots]$ instead. This has been possible in many cases \cite{McC4, McC5, McC8}, including for (\ref{thm_DMCsuper}). 
When $p \neq5$ and $h(p):=\biggfp{\frac{1}{5}}\biggfp{\frac{2}{5}}\biggfp{\frac{3}{5}}\biggfp{\frac{4}{5}}$, we get that \cite{McC4, McC5}, 
\begin{equation}\label{thm_DMCsuper2}
{_{4}F_3} \Biggl[ \begin{array}{cccc} \frac{1}{5} & \frac{2}{5} & \frac{3}{5} & \frac{4}{5}\\
\phantom{\frac{1}{5}} & 1 & 1 & 1 \end{array}
\bigg| \; 1 \Biggr]_{p-1}
\stackrel{\pmod {p^3}} \equiv
{_{4}G_{4}}
\biggl[ \begin{array}{cccc} \frac{1}{5} & \frac{2}{5} & \frac{3}{5} & \frac{4}{5} \\
 1 & 1 & 1 & 1 \end{array}
\Big| \; 1 \; \biggr]_p
- h(p) \cdot p
= b(p),
\end{equation}
which was the second of Rodriguez-Villegas' conjectures to be proven.


\section{Proofs}\label{sec_Proofs}

We begin by proving three transformations of ${_{n}G_{n}}[\cdots]_p$.

\begin{proof}[Proof of Theorem \ref{thm_3F2QuadTrans}]
Let $p$ be an odd prime. Taking $A=B=C= \phi_p$ in \cite[Theorem 4.28]{G} and reformulating in terms of Gauss sums and ${_{n+1}F_{n}}(\cdots)$, as defined in Definition \ref{def_F}, we get that, for $x \neq 1$,
\begin{multline}\label{for_Q3F2}
{_{3}F_{2}} {\biggl( \begin{array}{cccc} \phi & \phi & \phi \\
 \phantom{\phi} & \varepsilon & \varepsilon \end{array}
\Big| \; x \biggr)}_{p}\\
= 
\delta(1+x)  \cdot p \cdot \phi(-1)
+\frac{\phi(1-x)}{p-1}
 \sum_{\chi \in \widehat{\mathbb{F}_p^{*}}} 
\frac{g(\phi \chi^2) g(\phi \chi) g(\bar{\chi})^3}{g(\phi)^2} \; \chi \left(\frac{x}{(1-x)^2}\right).
\end{multline}
By definition of the Teichm\"{u}ller character we know that $\{\omega^j \mid 0 \leq j \leq p-2 \} =  \widehat{\mathbb{F}_p^{*}} $. Using Lemma \ref{lem_G_to_F} and the Gross-Koblitz formula, Theorem \ref{thm_GrossKoblitz}, to re-write equation \eqref{for_Q3F2} in the $p$-adic setting, we now have that
\begin{multline}\label{for_3GQuad1}
{_{3}G_{3}}
\biggl[ \begin{array}{cccc} \frac{1}{2} & \frac{1}{2} & \frac{1}{2} \\
 1 & 1 & 1 \end{array}
\Big| \; x^{-1} \; \biggr]_p
= 
\delta(1+x)  \cdot p \cdot \omega^{\frac{p-1}{2}}(-1)
-\frac{\omega^{\frac{p-1}{2}}(1-x)}{p-1}\\
\times \sum_{j=0}^{p-2}
(-p)^{-\lfloor{\frac{1}{2} -\frac{2j}{p-1}}\rfloor-\lfloor{\frac{1}{2} -\frac{j}{p-1}}\rfloor} 
\frac{\biggfp{\langle \frac{1}{2} -\frac{2j}{p-1}\rangle}\biggfp{\langle \frac{1}{2} -\frac{j}{p-1}\rangle}\biggfp{\frac{j}{p-1}}^3}{\biggfp{\frac{1}{2}}^2}
\; \omega^j \left(\frac{x}{(1-x)^2}\right).
\end{multline}
Applying (\ref{for_pGammaMult}) with $x =\langle \frac{1}{2} -\frac{2j}{p-1}\rangle$ and $m=2$ yields
\begin{equation*}
\Biggfp{\Big\langle \frac{1}{2} -\frac{2j}{p-1}\Big\rangle}
=\frac{\biggfp{\frac{1}{2}\langle \frac{1}{2} -\frac{2j}{p-1}\rangle} \biggfp{\frac{1}{2}\langle \frac{1}{2} -\frac{2j}{p-1}\rangle+\frac{1}{2}}}{\biggfp{\frac{1}{2}} \; \omega\left(2 ^{(1- \langle \frac{1}{2} -\frac{2j}{p-1}\rangle)(1-p)}\right)}.
\end{equation*}
By considering $j$ in the intervals $[0, \lfloor{\frac{p-1}{4}}\rfloor], (\lfloor{\frac{p-1}{4}}\rfloor, \lfloor{\frac{3(p-1)}{4}}\rfloor]$ and $(\lfloor{\frac{3(p-1)}{4}}\rfloor, p-2]$ it straightforward to verify that
\begin{equation*}
\Biggfp{\frac{1}{2}\Big\langle \frac{1}{2} -\frac{2j}{p-1}\rangle} \Biggfp{\frac{1}{2}\Big\langle \frac{1}{2} -\frac{2j}{p-1}\Big\rangle+\frac{1}{2}}=
\Biggfp{\Big\langle \frac{1}{4} -\frac{j}{p-1}\Big\rangle}\Biggfp{\Big\langle \frac{3}{4} -\frac{j}{p-1}\Big\rangle}
\end{equation*}
and
\begin{equation*}
\omega\left(2 ^{(1- \langle \frac{1}{2} -\frac{2j}{p-1}\rangle)(1-p)}\right)=\omega^{\frac{p-1}{2}}(2) \; \bar{\omega}^j(4).
\end{equation*}
Therefore
\begin{equation}\label{for_gammap2j}
\Biggfp{\Big\langle \frac{1}{2} -\frac{2j}{p-1}\Big\rangle}
=\frac{\biggfp{\langle \frac{1}{4} -\frac{j}{p-1}\rangle} \biggfp{\langle \frac{3}{4} -\frac{j}{p-1}\rangle}}{\biggfp{\frac{1}{2}} \; \omega^{\frac{p-1}{2}}(2) \; \bar{\omega}^j(4)}.
\end{equation}
Note also that
\begin{equation}\label{for_floor2j}
\Big\lfloor{\frac{1}{2} -\frac{2j}{p-1}}\Big\rfloor = \Big\lfloor{\frac{1}{4} -\frac{j}{p-1}}\Big\rfloor + \Big\lfloor{\frac{3}{4} -\frac{j}{p-1}}\Big\rfloor.
\end{equation}
Accounting for (\ref{for_gammap2j}) and (\ref{for_floor2j}) in (\ref{for_3GQuad1}) yields
\begin{multline*}\label{for_3GQuad2}
{_{3}G_{3}}
\biggl[ \begin{array}{cccc} \frac{1}{2} & \frac{1}{2} & \frac{1}{2} \\
 1 & 1 & 1 \end{array}
\Big| \; x^{-1} \; \biggr]_p
= 
\delta(1+x)  \cdot p \cdot \omega^{\frac{p-1}{2}}(-1)
-\frac{\omega^{\frac{p-1}{2}}(2(1-x))}{p-1}\\
\times \sum_{j=0}^{p-2}
\bar{\omega}^j \left(\frac{(1-x)^2}{4x}\right)
\frac{\biggfp{\langle \frac{1}{4} -\frac{j}{p-1}\rangle} \biggfp{\langle \frac{3}{4} -\frac{j}{p-1}\rangle}\biggfp{\langle \frac{1}{2} -\frac{j}{p-1}\rangle}\biggfp{\frac{j}{p-1}}^3}{\biggfp{\frac{1}{2}}^3}\\
\cdot (-p)^{-\lfloor{\frac{1}{4} -\frac{j}{p-1}}\rfloor - \lfloor{\frac{3}{4} -\frac{j}{p-1}}\rfloor-\lfloor{\frac{1}{2} -\frac{j}{p-1}}\rfloor}  .
\end{multline*}
Noting that $\bar{\omega}^j(-1) = (-1)^j$ and $\biggfp{\frac{1}{2}}^4=1$, by (\ref{for_pGammaOneMinus}), yields the result.
\end{proof}

\begin{proof}[Proof of Theorem \ref{thm_GReduction}]
Let $p$ be an odd prime. By definition
\begin{multline*}
{_{n+1}G_{n+1}}
\biggl[ \begin{array}{ccccc} a_1 & a_2 & \dotsc & a_n & \frac{1}{2} \\
 1 & 1 & \dotsc & 1 & 1 \end{array}
\Big| \; s \; \biggr]_p
= \frac{-1}{p-1}  \sum_{j=0}^{p-2} 
(-1)^{j(n+1)}\;
\bar{\omega}^j(s)\\
\times \prod_{i=1}^{n} 
\frac{\biggfp{\langle a_i -\frac{j}{p-1}\rangle} \, \biggfp{\frac{j}{p-1}}}{\biggfp{\langle a_i \rangle}}
(-p)^{-\lfloor{\langle a_i \rangle -\frac{j}{p-1}}\rfloor}\\
\times
\frac{\biggfp{\langle \frac{1}{2} -\frac{j}{p-1}\rangle}\biggfp{\frac{j}{p-1}}}{\biggfp{\frac{1}{2}}}
(-p)^{-\lfloor{\frac{1}{2} -\frac{j}{p-1}}\rfloor}. 
\end{multline*}
Applying Lemma \ref{NonQuad_Sum} to the expression on the last line above and tidying up yields
\begin{multline*}
{_{n+1}G_{n+1}}
\biggl[ \begin{array}{ccccc} a_1 & a_2 & \dotsc & a_n & \frac{1}{2} \\
 1 & 1 & \dotsc & 1 & 1 \end{array}
\Big| \; s \; \biggr]_p\\
=
-\sum_{t=2}^{p-1} \; 
{_{n}G_{n}}
\biggl[ \begin{array}{cccc} a_1 & a_2 & \dotsc & a_n \\
 1 & 1 & \dotsc & 1 \end{array}
\Big| \; st^{-1} \; \biggr]_p
\cdot \omega^{\frac{p-1}{2}} (t(t-1))
\end{multline*}
Letting $t \rightarrow t^{-1}$ and noting that $\omega^{\frac{p-1}{2}}(t^2)=1$ completes the proof.
\end{proof}

\begin{proof}[Proof of Theorem \ref{thm_GInversion}]
Let $p$ be an odd prime and let $2 \leq n \in \mathbb{Z}$. By definition
\begin{multline*}
{_{n}G_{n}}
\biggl[ \begin{array}{cccccc} \frac{1}{2} &  \frac{1}{2} & \dotsc &  \frac{1}{2} & \frac{1}{4} & \frac{3}{4} \\
 1 & 1 & \dotsc & 1 & 1 & 1 \end{array}
\Big| \; s \; \biggr]_p\\
= \frac{-1}{p-1}  \sum_{j=0}^{p-2} 
(-1)^{jn}\;
\bar{\omega}^j(s)
\frac{\biggfp{\langle \frac{1}{2} -\frac{j}{p-1}\rangle}^{n-2} \,\biggfp{\langle \frac{1}{4} -\frac{j}{p-1}\rangle}\biggfp{\langle \frac{3}{4} -\frac{j}{p-1}\rangle}\biggfp{\langle \frac{j}{p-1}\rangle}^{n}}
{\biggfp{\frac{1}{2}}^{n-2}\biggfp{\frac{1}{4} }\biggfp{ \frac{3}{4}}}\\
\times
(-p)^{-(n-2)\lfloor{ \frac{1}{2}  -\frac{j}{p-1}}\rfloor -\lfloor{ \frac{1}{4} -\frac{j}{p-1}}\rfloor -\lfloor{ \frac{3}{4}  -\frac{j}{p-1}}\rfloor - n \lfloor{\frac{j}{p-1}}\rfloor}.
\end{multline*}

\noindent Letting $j=\frac{p-1}{2} - k$ this becomes
\begin{multline*}
{_{n}G_{n}}
\biggl[ \begin{array}{cccccc} \frac{1}{2} &  \frac{1}{2} & \dotsc &  \frac{1}{2} & \frac{1}{4} & \frac{3}{4} \\
 1 & 1 & \dotsc & 1 & 1 & 1 \end{array}
\Big| \; s \; \biggr]_p\\
\shoveleft = \frac{-1}{p-1}  \sum_{k=1-\frac{p-1}{2}}^{\frac{p-1}{2}} 
(-1)^{n\left(\frac{p-1}{2}-k \right)}\;
\bar{\omega}^{\frac{p-1}{2}-k}(s)\\
\times
\frac{\biggfp{\langle \frac{k}{p-1}\rangle}^{n-2} \,\biggfp{\langle -\frac{1}{4} +\frac{k}{p-1}\rangle}\biggfp{\langle \frac{1}{4} +\frac{k}{p-1}\rangle}\biggfp{\langle\frac{1}{2} - \frac{k}{p-1}\rangle}^{n}}
{\biggfp{\frac{1}{2}}^{n-2}\biggfp{\frac{1}{4} }\biggfp{ \frac{3}{4}}}\\
\times
(-p)^{-(n-2)\lfloor{ \frac{k}{p-1}}\rfloor -\lfloor{ -\frac{1}{4} +\frac{k}{p-1}}\rfloor -\lfloor{ \frac{1}{4}  +\frac{k}{p-1}}\rfloor - n \lfloor{\frac{1}{2} - \frac{k}{p-1}}\rfloor}.
\end{multline*}

\noindent Note the summand is invariant under $k \rightarrow k+ (p-1)$ so we can take the limits of summation to be from $k=0$ to $k=p-2$. Therefore,
\begin{multline*}
{_{n}G_{n}}
\biggl[ \begin{array}{cccccc} \frac{1}{2} &  \frac{1}{2} & \dotsc &  \frac{1}{2} & \frac{1}{4} & \frac{3}{4} \\
 1 & 1 & \dotsc & 1 & 1 & 1 \end{array}
\Big| \; s \; \biggr]_p\\
\shoveleft = \frac{-1}{p-1}  \sum_{k=0}^{p-2} 
(-1)^{n\left(\frac{p-1}{2}-k \right)}\;
\bar{\omega}^{\frac{p-1}{2}-k}(s) \\
\times
\frac{\biggfp{\langle \frac{k}{p-1}\rangle}^{n-2} \,\biggfp{\langle -\frac{1}{4} +\frac{k}{p-1}\rangle}\biggfp{\langle \frac{1}{4} +\frac{k}{p-1}\rangle}\biggfp{\langle\frac{1}{2} - \frac{k}{p-1}\rangle}^{n}}
{\biggfp{\frac{1}{2}}^{n-2}\biggfp{\frac{1}{4} }\biggfp{ \frac{3}{4}}}\\
\times
(-p)^{-(n-2)\lfloor{ \frac{k}{p-1}}\rfloor -\lfloor{ -\frac{1}{4} +\frac{k}{p-1}}\rfloor -\lfloor{ \frac{1}{4}  +\frac{k}{p-1}}\rfloor - n \lfloor{\frac{1}{2} - \frac{k}{p-1}}\rfloor}\\
\shoveleft= \frac{-p \; \bar{\omega}^{\frac{p-1}{2}}((-1)^{n+1} s)}{p-1}  \sum_{k=0}^{p-2} 
(-1)^{nk}\;
 \bar{\omega}^{k}(s^{-1})\\
\times
\frac{\biggfp{\langle\frac{1}{2} - \frac{k}{p-1}\rangle}^{n} \, \biggfp{\langle \frac{k}{p-1}\rangle}^{n-2} \,\biggfp{\langle \frac{1}{4} +\frac{k}{p-1}\rangle}\biggfp{\langle \frac{3}{4} +\frac{k}{p-1}\rangle}}
{\biggfp{\frac{1}{2}}^{n}\,\biggfp{\frac{1}{4} }\biggfp{ \frac{3}{4}}}\\
\times
(-p)^{- n \lfloor{\frac{1}{2} - \frac{k}{p-1}}\rfloor -(n-2)\lfloor{ \frac{k}{p-1}}\rfloor  -\lfloor{ \frac{1}{4}  +\frac{k}{p-1}}\rfloor -\lfloor{ \frac{3}{4} +\frac{k}{p-1}}\rfloor }\\
=
 p \cdot \omega^{\frac{p-1}{2}}((-1)^{n+1}s) \cdot
{_{n}G_{n}}
\biggl[ \begin{array}{cccccc} \frac{1}{2} &  \frac{1}{2} & \dotsc &  \frac{1}{2} & \frac{1}{2} & \frac{1}{2} \\
 1 & 1 & \dotsc & 1 & \frac{1}{4} & \frac{3}{4} \end{array}
\Big| \; s^{-1} \; \biggr]_p,
\end{multline*}
using the facts that $\biggfp{\frac{1}{2}}^2 = - \omega^{\frac{p-1}{2}}(-1)$ and $(-1)^{\frac{p-1}{2}} = \phi(-1) = \omega^{\frac{p-1}{2}}(-1)$.
\end{proof}

We now prove Theorem \ref{thm_GMainId} by using the transformations in Theorems \ref{thm_3F2QuadTrans}-\ref{thm_GInversion}, together with Lemma \ref{Quad_Sum}.

\begin{proof}[Proof of Theorem \ref{thm_GMainId}]
Let $p$ be an odd prime. By letting $t \rightarrow t^{-1}$ in Theorem \ref{thm_GReduction}, we see that
\begin{equation*}
{_{4}G_{4}}
\biggl[ \begin{array}{cccc} \frac{1}{2} & \frac{1}{2} & \frac{1}{2} & \frac{1}{2} \\
 1 & 1 & 1 & 1 \end{array}
\Big| \; 1 \; \biggr]_p
=
-
\sum_{t=2}^{p-1}
{_{3}G_{3}}
\biggl[ \begin{array}{cccc} \frac{1}{2} & \frac{1}{2} & \frac{1}{2} \\
 1 & 1  & 1 \end{array}
\Big| \; t^{-1} \; \biggr]_p
\cdot \; \omega^{\frac{p-1}{2}}(t(t-1)).
\end{equation*}

\noindent Applying Theorem \ref{thm_3F2QuadTrans} to the right-hand side gives us
\begin{multline*}
{_{4}G_{4}}
\biggl[ \begin{array}{cccc} \frac{1}{2} & \frac{1}{2} & \frac{1}{2} & \frac{1}{2} \\
 1 & 1 & 1 & 1 \end{array}
\Big| \; 1 \; \biggr]_p\\
=
-p \cdot \omega^{\frac{p-1}{2}}(-2)
- s(p) \cdot \omega^{\frac{p-1}{2}}(-2) \cdot
\sum_{t=2}^{p-1}
{_{3}G_{3}}
\biggl[ \begin{array}{cccc} \frac{1}{4} & \frac{3}{4} & \frac{1}{2} \\
 1 & 1 & 1 \end{array}
\Big| \; -\frac{(1-t)^2}{4t}  \; \biggr]_p
\cdot \; \omega^{\frac{p-1}{2}}(t).
\end{multline*}

\noindent Now using Theorem \ref{thm_GInversion} and Lemma \ref {Quad_Sum} we see that
\begin{align*}
{_{4}G_{4}}&
\biggl[ \begin{array}{cccc} \frac{1}{2} & \frac{1}{2} & \frac{1}{4} & \frac{3}{4} \\
 1 & 1 & 1 & 1 \end{array}
\Big| \; 1 \; \biggr]_p\\
&=
p \cdot \omega^{\frac{p-1}{2}}(-1) \cdot
{_{4}G_{4}}
\biggl[ \begin{array}{cccc} \frac{1}{2} & \frac{1}{2} & \frac{1}{2} & \frac{1}{2} \\
 1 & 1 & \frac{1}{4} & \frac{3}{4} \end{array}
\Big| \; 1 \; \biggr]_p\\
&=
\frac{-p \cdot \omega^{\frac{p-1}{2}}(-1)}{p-1}  \sum_{j=0}^{p-2}
\frac
{\biggfp{\langle \frac{1}{2} -\frac{j}{p-1}\rangle}^4 \, \biggfp{ \frac{j}{p-1}}^2 \, \biggfp{\langle \frac{1}{4} +\frac{j}{p-1}\rangle}\biggfp{\langle \frac{3}{4} +\frac{j}{p-1}\rangle}}
{\biggfp{ \frac{1}{2} }^4 \, \biggfp{ \frac{1}{4} } \biggfp{ \frac{3}{4} }}\\
& \qquad \qquad \qquad \qquad\qquad \qquad \qquad \qquad \qquad \qquad \qquad  \times 
(-p)^{-4\lfloor{ \frac{1}{2} -\frac{j}{p-1}}\rfloor - \lfloor{ \frac{3}{4} +\frac{j}{p-1}}\rfloor - \lfloor{\frac{1}{4}  +\frac{j}{p-1}}\rfloor}
\\
&=
\frac{p \cdot \omega^{\frac{p-1}{2}}(-1)}{p-1} 
\left[
 \sum_{j=1}^{p-2} 
\frac
{\biggfp{\langle \frac{1}{2} -\frac{j}{p-1}\rangle}^3 \, \biggfp{ \frac{j}{p-1}} \, \biggfp{\langle \frac{1}{4} +\frac{j}{p-1}\rangle}\biggfp{\langle \frac{3}{4} +\frac{j}{p-1}\rangle}}
{\biggfp{ \frac{1}{2} }^3 \, \biggfp{ \frac{1}{4} } \biggfp{ \frac{3}{4} }}
\right.\\
& \left.  \quad \qquad \qquad \qquad \qquad \qquad  \times 
(-p)^{-3\lfloor{ \frac{1}{2} -\frac{j}{p-1}}\rfloor - \lfloor{ \frac{3}{4} +\frac{j}{p-1}}\rfloor - \lfloor{\frac{1}{4}  +\frac{j}{p-1}}\rfloor}
\sum_{t=2}^{p-1} \; \bar{\omega}^j\left(\frac{4(1-t)}{t^2}\right)
-1
\right]
\\
&=
-p \cdot \omega^{\frac{p-1}{2}}(-1) \cdot
\sum_{t=2}^{p-1}
{_{3}G_{3}}
\biggl[ \begin{array}{ccc} \frac{1}{2} & \frac{1}{2} & \frac{1}{2} \\
 1 & \frac{1}{4} & \frac{3}{4} \end{array}
\Big| \; -\frac{4(1-t)}{t^2} \; \biggr]_p
-p \cdot \omega^{\frac{p-1}{2}}(-1) \\
&=
-p \cdot \omega^{\frac{p-1}{2}}(-1) \cdot
\sum_{t=2}^{p-1}
{_{3}G_{3}}
\biggl[ \begin{array}{ccc} \frac{1}{2} & \frac{1}{2} & \frac{1}{2} \\
 1 & \frac{1}{4} & \frac{3}{4} \end{array}
\Big| \; -\frac{4t}{(1-t)^2} \; \biggr]_p
-p \cdot \omega^{\frac{p-1}{2}}(-1) \\
&=
-\sum_{t=2}^{p-1}
{_{3}G_{3}}
\biggl[ \begin{array}{ccc} \frac{1}{2} & \frac{1}{4} & \frac{3}{4} \\
 1 & 1 & 1 \end{array}
\Big| \; -\frac{(1-t)^2}{4t} \; \biggr]_p
\omega^{\frac{p-1}{2}}(t)
-p \cdot \omega^{\frac{p-1}{2}}(-1), 
\end{align*}
where, in the second last step, we let $t \rightarrow 1-t$.
\noindent We note that $\omega^{\frac{p-1}{2}}(2) = s(p)$. Therefore
\begin{align*}
{_{4}G_{4}} &
\biggl[ \begin{array}{cccc} \frac{1}{2} & \frac{1}{2} & \frac{1}{4} & \frac{3}{4} \\
 1 & 1 & 1 & 1 \end{array}
\Big| \; 1 \; \biggr]_p 
- s(p) \cdot p\\
&=
-\sum_{t=2}^{p-1}
{_{3}G_{3}}
\biggl[ \begin{array}{ccc} \frac{1}{2} & \frac{1}{4} & \frac{3}{4} \\
 1 & 1 & 1 \end{array}
\Big| \; -\frac{(1-t)^2}{4t} \; \biggr]_p
\omega^{\frac{p-1}{2}}(t)
-p \cdot \left(s(p) + \omega^{\frac{p-1}{2}}(-1) \right) \\
&=\omega^{\frac{p-1}{2}}(-1)
\left[ \,
{_{4}G_{4}}
\biggl[ \begin{array}{cccc} \frac{1}{2} & \frac{1}{2} & \frac{1}{2} & \frac{1}{2} \\
 1 & 1 & \dotsc & 1 \end{array}
\Big| \; 1 \; \biggr]_p
-p \,
\right].
\end{align*}

\end{proof}

We conclude this section by using Theorem \ref{thm_GMainId} to verify Theorem \ref{thm_GtoC}. Then, we apply Theorems \ref{thm_4G1} and  \ref{thm_GtoC} to prove the supercongruence in Theorem \ref{thm_RV3}.

\begin{proof}[Proof of Theorem \ref{thm_GtoC}]
Combining (\ref{thm_AO}) and Lemma \ref{lem_G_to_F} we know that, for all odd primes $p$,
\begin{equation*}
{_{4}G_{4}}
\biggl[ \begin{array}{cccc} \frac{1}{2} & \frac{1}{2} & \frac{1}{2} & \frac{1}{2} \\
 1 & 1 & 1 & 1 \end{array}
\Big| \; 1 \; \biggr]_p
-p=a(p).
\end{equation*}
Now, $f_2$ is a quadratic twist of $f_1$ with $c_p = \phi_p(-1) a_p$. To see this, we 
consider the quadratic twist  $f_{1,\psi} = \sum_{n=1}^\infty \psi(n)a(n)q^n$ of $f_1$ by the character $\psi(\cdot) = (\tfrac{-4}{\cdot})$, the Kronecker symbol.
A priori, $f_{1,\psi} \in S_4(\Gamma_0(32))$ by \cite[\S3 Prop.~17(b)]{Ko1}; however, the proof of \cite[\S3 Prop.~17(b)]{Ko1} reveals in fact that $f_{1,\psi} \in S_4(\Gamma_0(\textup{lcm}(8,4^2))=S_4(\Gamma_0(16))$. Using Sage Mathematics Software~\cite{Sage}, one easily computes the dimensions $\textup{Dim}(S_4(\Gamma_0(8)))=1$ and  $\textup{Dim}(S_4(\Gamma_0(16)))=3$. Therefore, $f_1(z)$ and $f_1(2z)$ form a basis for $S_4^{\textup{old}}(\Gamma_0(16))$ and so $f_{1,\psi} \in S_4^{\textup{new}}(\Gamma_0(16))$. 
Therefore $f_{1,\psi}$ is a multiple of $f_2$, the unique newform in $S_4(\Gamma_0(16))$, and by comparing coefficients we see that $f_{1,\psi}=f_2$. As a result,  $c_p = \psi(p) a_p = \phi_p(-1) a_p$.

Noting that $\phi_p(-1) = \omega^{(p-1)/2}(-1)$, it therefore suffices to prove
\begin{equation*}
{_{4}G_{4}}
\biggl[ \begin{array}{cccc} \frac{1}{2} & \frac{1}{2} & \frac{1}{4} & \frac{3}{4} \\
 1 & 1 & 1 & 1 \end{array}
\Big| \; 1 \; \biggr]_p
- s(p) \cdot p
= \omega^{(p-1)/2}(-1)
\left[
{_{4}G_{4}}
\biggl[ \begin{array}{cccc} \frac{1}{2} & \frac{1}{2} & \frac{1}{2} & \frac{1}{2} \\
 1 & 1 & 1 & 1 \end{array}
\Big| \; 1 \; \biggr]_p
-p
\right].
\end{equation*}
This has been proved in Theorem \ref{thm_GMainId}.
\end{proof}

\begin{proof}[Proof of Theorem \ref{thm_RV3}]
Appealing to Theorem \ref{thm_4G1} with $d_1=2$, $d_2=4$, it suffices to prove
\begin{equation*}
{_{4}G_{4}}
\biggl[ \begin{array}{cccc} \frac{1}{2} & \frac{1}{2} & \frac{1}{4} & \frac{3}{4} \\
 1 & 1 & 1 & 1 \end{array}
\Big| \; 1 \; \biggr]_p
- s(p) \cdot p
=c(p),
\end{equation*}
where $s(p)=\biggfp{\frac{1}{4}}\biggfp{\frac{3}{4}}\biggfp{\frac{1}{2}}^2$.
This has been proved in Theorem \ref{thm_GtoC}.
\end{proof}


\section{Concluding Remarks}\label{sec_cr}
The proof of the supercongruence in this paper relies on the fact that the modular form in question is a twist of another modular form, which appears in one of the other known supercongruences. This allows us to link the proof of the two supercongruences via an identity for  ${_{n}G_{n}}[\cdots]$. Unfortunately, these are the only two modular forms in Rodriguez-Villegas' conjectures which are linked in this way. Thus, it does not appear that discovering other identities for ${_{n}G_{n}}[\cdots]$, which are similar to Theorem \ref{thm_GMainId}, will reduce the burden of proof for establishing the remaining 11 conjectures.


\end{document}